\documentclass[a4paper,11pt,centertags,draft]{amsart}
\usepackage[latin1]{inputenc}
\usepackage[english]{babel}
\usepackage{cite}
\usepackage{verbatim}
\usepackage{color}
\usepackage{hyperref}
\usepackage{palatino,mathpazo}
\usepackage{amsmath,amsthm}
\usepackage{comment}

\usepackage[top=3.2cm, bottom=3.2cm, left=2.8cm,right=2.8cm]{geometry}

\newcommand{\diesis}{^\#}
\newtheorem{theo}{Theorem}[section]
\newtheorem{lemma}{Lemma}[section]

\theoremstyle{definition}
\newtheorem{definiz}{Definition}[section]
\newtheorem{rem}{Remark}[section]

\numberwithin{equation}{section}

\newcommand{\R}{\mathbb R}

\newcommand{\de}{\partial}

\begin{document}
\title[Sharp bounds for the first eigenvalue and the torsional
rigidity]{Sharp bounds for the first eigenvalue and the torsional
  rigidity related to some anisotropic operators}
\author[F. Della Pietra, N. Gavitone]{
  Francesco Della Pietra and Nunzia Gavitone
}
\address{Francesco Della Pietra \\
Universit\`a degli studi del Molise \\
Dipartimento di Bioscienze e Territorio -
Divisione di Fisica, Informatica e Matematica\\
Via Duca degli Abruzzi \\
86039 Termoli (CB), Italia.
}
\email{francesco.dellapietra@unimol.it}

 \address{
Nunzia Gavitone \\
Universit\`a degli studi di Napoli ``Federico II''\\
Dipartimento di Matematica e Applicazioni ``R. Caccioppoli''\\
80126 Napoli, Italia.
}
\email{nunzia.gavitone@unina.it}
\maketitle
\date{\today}
\begin{abstract}
In this paper we prove a sharp upper bound for the first Dirichlet
eigenvalue of a class of nonlinear elliptic operators which includes
the operator $\Delta_p u =\sum_i \frac{\partial}{\partial
  x_i} (|\nabla u|^{p-2}\! \frac{\partial u}{\partial x_i})$, that is
the $p$-Laplacian, and $\tilde \Delta_p u= \sum_i
\frac{\partial}{\partial x_i} (|\frac{\partial u}{\partial x_i}|^{p-2}
\frac{\partial u}{\partial x_i})$, namely the
pseudo-$p$-Laplacian. Moreover we prove a stability result by means of
a
suitable isoperimetric deficit. Finally, we give a 
sharp lower bound for the anisotropic $p$-torsional rigidity.
\end{abstract}
\section{Introduction}
Let $H:\R^n\rightarrow [0,+\infty[$, $n\ge 2$, be a convex,
$1$-homogeneous and  $C^1(\mathbb R^2\setminus\{0\})$ function. 
Here  we  deal with operators of the form
\begin{equation*}
    \mathcal Q_p u:= \sum_{i=1}^{n} \frac{\de}{\de x_i}
  \big(H(Du)^{p-1}H_{\xi_i}(Du)\big),
\end{equation*}
with $1<p<\infty$. In general, $\mathcal Q_p $ is highly nonlinear, and
extends some well-known classes of operators. In particular, for
$H(\xi)= ( \sum_k |\xi_k|^r )^{1/r}$, $r>1$, $\mathcal Q_p$ becomes
\begin{equation*}
  \mathcal Q_p u= \sum_{i=1}^{n} \frac{\de}{\de x_i} \left(
    \left( \sum_{k=1}^{n} \left| \frac{\de u}{\de x_k} \right|^r
    \right)^{(p-r)/r} \left| \frac{\de u}{\de x_i}
    \right|^{r-2}\frac{\de u}{ \de x_i}\right).
\end{equation*}
Note that for $r=2$, it coincides with the usual $p$-Laplace
operator, while for $r=p$ it is the so-called pseudo-$p$-Laplace
operator.

This kind of operator has been studied in several papers
(see for instance \cite{aflt}, \cite{ciasal}, \cite{dpg2},
\cite{dpg3}, \cite{ferkaw} for $p=2$, and \cite{bfk}, \cite{bkj06},
\cite{dpg4}, \cite{kn08} for $1<p<\infty$).

In this paper  we consider the eigenvalue problem associated to
$\mathcal Q_p$, namely
\begin{equation}\label{eig:intro}
  \left\{
    \begin{array}{ll}
      -\mathcal Q_p u= \lambda |u|^{p-2} u &
      \text{in }\Omega, \\ [.2cm]
      u = 0 & \text{on }\de\Omega,
    \end{array}
  \right.
\end{equation}
where $\Omega$ is a bounded convex open set of $\R^n$.

Our aim is to extend to the first eigenvalue $\lambda_p(\Omega)$ of
the problem \eqref{eig:intro}, some well-known estimates involving the
first eigenvalue $\lambda_{1,2}(\Omega)$  of the Dirichlet-Laplacian. In
this case, a classical result by P\'olya contained in \cite{po61}
gives the following upper bound for $\lambda_{1,2}(\Omega)$ in the plane:
\begin{equation}
\label{pol}
\lambda_{1,2}(\Omega) \le \left(\frac{\pi}{2}\right)^2
\frac{P^2(\Omega)}{|\Omega|^2},
\end{equation}
where $\Omega$ is a bounded simply connected  open set of $\R^2$ and
$|\Omega|$ and $P(\Omega)$ denote respectively the Lebesgue measure
and the Euclidean perimeter of $\Omega$.

Another classical result, due to Payne and Weinberger (see
\cite{pw61}), gives an  upper bound of $\lambda_{1,2}(\Omega)$
by means of an isoperimetric deficit. More precisely, if $\Omega$ is
a bounded, simply connected, open set of $\R^2$, denoting by
$\Omega\diesis$ the ball centered at the origin with the same measure
as $\Omega$, then 
\begin{equation}
  \label{pw}
  \lambda_{1,2}(\Omega) \le \lambda_{1,2}(\Omega\diesis) \left[1+C
    \left(\frac{P^2(\Omega)}{4\pi |\Omega|}-1\right) \right],
\end{equation}
where $C$ is a universal sharp constant, which can be explicitly
determined. Hence \eqref{pw}, together with the Faber-Krahn
inequality, that is 
\[
\lambda_{1,2}(\Omega\diesis) \le \lambda_{1,2}(\Omega),
\]
gives a stability estimate for $\lambda_{1,2}(\Omega)$, namely
\begin{equation}
  \label{stab1}
  0\le  \frac{ \lambda_{1,2}(\Omega) -
    \lambda_{1,2}(\Omega\diesis)}{\lambda_{1,2}(\Omega\diesis)} \le C
    \left(\frac{P^2(\Omega)}{4\pi |\Omega|}-1\right).
  \end{equation}
  Recently, in \cite{bnt10} the authors have proved
  an estimate similar to \eqref{stab1}, in any dimension, for
  the first Dirichlet-eigenvalue $\lambda_{1,p}(\Omega)$ of the
  $p$-Laplacian. In particular, they prove that, if $\Omega$ is a
  bounded convex open set of $\R^n$, then 
  \begin{equation}
  \label{stab2}
  \frac{ \lambda_{1,p}(\Omega) -
    \lambda_{1,p}(\hat\Omega)}{\lambda_{1,p}(\Omega)} \le C(n,p,\Omega)
    \left(1-\frac{n^{\frac{n}{n-1}}\omega_n^{\frac{1}{n-1}}|\Omega|}
      {P(\Omega)^{\frac{n}{n-1}}}\right),
  \end{equation}
  where $\hat\Omega$ is the ball centered at
  the origin with the same perimeter as $\Omega$. As matter of fact,
  being $\lambda_{1,p}(\hat\Omega)\le
  \lambda_{1,p}(\Omega\diesis)\le \lambda_{1,p}(\Omega)$, 
  we have that the left-hand side of \eqref{stab2} is nonnegative.

We stress that in \cite{carlo} the author gives, in the case $p=2$, the
best value of the constant $C(n,2,\Omega)$ in \eqref{stab2}. 

We recall that this kind of estimates have been studied also for a
class of fully nonlinear elliptic problems, involving the so-called
$k$-Hessian operators (see \cite{dpg5}). 

Our aim is to prove similar results for the first eigenvalue
$\lambda_p(\Omega)$ of \eqref{eig:intro} when $\Omega$ is a bounded,
convex, open set of $\R^n$. The first result we show is a P\'olya-type
estimate, namely that
 \begin{equation*}
   \label{polintro}
   \lambda_p(\Omega) \le  \left(\frac{\pi_p}{2}\right)^p
   \frac{P_H(\Omega)^{p}}{|\Omega|^{p}}.
 \end{equation*}
Here $P_H(\Omega)$ denotes the anisotropic perimeter of $\Omega$
relative to the norm $H$ (see Section 2 for the precise definitions).


The second result we prove regards a stability estimate.
If we denote by $H^o(\xi)$ the polar function of $H$, and by $\tilde
\Omega$ the level set of $H^o$ with the same anisotropic perimeter as
$\Omega$, our result reads as follows. 
\begin{equation}
  \label{stab3}
     \frac{ \lambda_p(\Omega) -
   \lambda_p(\tilde \Omega)}{\lambda_p(\Omega)} \le
C_\Omega
\left(1-\frac{n^{\frac{n}{n-1}}\kappa_n^{\frac{1}{n-1}}|\Omega|}
  {P_H(\Omega)^{\frac{n}{n-1}}} \right).
\end{equation}
Again the above inequality, in conjunction with the Faber-Krahn
inequality for $\mathcal Q_p$ (see \cite{bfk} and Section 2.2), gives a
stability estimate. 
 
Last part of the paper is devoted to give a lower bound to the
an\-isotropic $p$-torsional rigidity, namely the number
$\tau_p(\Omega)>0$ such that  
\begin{equation*}
  \tau_p(\Omega)=\int_\Omega H(Du_p)^p dx = \int_\Omega u_p dx,
\end{equation*}
where $u_p\in W_0^{1,p}(\Omega)$ is the unique solution of
\begin{equation*}
  \left\{
  \begin{array}{ll}
    -\mathcal Q_p u = 1 &\text{in }\Omega,\\
    u=0 &\text{on }\de\Omega.
  \end{array}
  \right.
\end{equation*}

In the Euclidean case, in \cite{po61} and \cite{cfg02} the
authors prove a lower bound for the $p$-torsional 
rigidity respectively for $p=2$ and for a general $1<p<+\infty$. Such
results can be summarized in the estimate 
\begin{equation}
  \label{gazz2}
 \tau_p(\Omega)\ge \frac{p-1}{2p-1}
 \frac{|\Omega|^{\frac{2p-1}{p-1}}}{P(\Omega)^{\frac{p}{p-1}}},
 \end{equation}
where $\Omega \subset \R^2$ is a bounded convex open set.

Our aim is also to prove a similar result in  the anisotropic case,
for a general norm $H$ of $\R^n$ (see Theorem \ref{teotor}).

In order to obtain the quoted estimates in the case $H(\xi)=|\xi|$,
the proof makes use of a particular class of test functions, depending
on the Euclidean distance to the boundary, introduced in \cite{po61}
and nowadays known as web functions (see for 
example \cite{cgweb}). It seems be natural to consider, in our case,
an analogous class of web functions, which depend on a distance
to the boundary defined according to the general norm $H$.  

The paper is organized as follows. In Section 2, we recall some useful
properties of the functions $H$ and $H^o$ and some basic definitions of
the anisotropic perimeter and of convex analysis. Moreover, we recall
the main properties of the first eigenvalue and of
the $p$-torsional rigidity of $\mathcal Q_p$. Then, in Section 
3 we prove some preliminary results involving the anisotropic distance
to the boundary, necessary to obtain the main theorems. Finally, in
Sections 4 and 5 we state precisely the main results and give the
proofs. 

\section{Notation and preliminaries}
\subsection{Anisotropic norm and perimeter}
\label{anintro}
Let $H:\R^n\rightarrow [0,+\infty[$, $n\ge 2$, be a convex
$C^1(\mathbb R^2\setminus\{0\})$ function such that
\begin{equation}\label{eq:omo}
  H(t\xi)= |t| H(\xi), \quad \forall \xi \in \R^n,\; \forall t \in
  \R.
\end{equation}
Moreover, suppose that there exist two positive constants $\alpha
\le \beta$ such that
\begin{equation}\label{eq:lin}
  \alpha|\xi| \le H(\xi) \le \beta|\xi|,\quad \forall \xi\in \R^n.
\end{equation}
 
We define the polar function $H^o\colon\R^n \rightarrow [0,+\infty[$ 
of $H$ as
\begin{equation}
  \label{eq:pol}
H^o(v)=\sup_{\xi \ne 0} \frac{\xi\cdot v}{H(\xi)}  
\end{equation}
where $\langle\cdot,\cdot\rangle$ is the usual scalar product of
$\R^n$. It is easy to verify that also $H^o$ is a convex function
which satisfies properties \eqref{eq:omo} and
\eqref{eq:lin}. Furthermore, 
\[
H(v)=\sup_{\xi \ne 0} \frac{\xi\cdot v}{H^o(\xi)}.
\]
The set
\[
\mathcal W = \{  \xi \in \R^n \colon H^o(\xi)< 1 \}
\]
is the so-called Wulff shape centered at the origin. We put
$\kappa_n=|\mathcal W|$, where $|\mathcal W|$ denotes the Lebesgue measure
of $\mathcal W$. More generally, we denote with $\mathcal W_r(x_0)$
the set $r\mathcal W+x_0$, that is the Wulff shape centered at $x_0$
with measure $\kappa_nr^n$, and $\mathcal W_r(0)=\mathcal W_r$.

The following properties of $H$ and $H^o$ hold true
(see for example \cite{bp}):
\begin{gather}
 H(\nabla H^o(\xi))=H^o(\nabla H(\xi))=1,\quad \forall \xi \in
\R^n\setminus \{0\}, \label{eq:H1} \\
H^o(\xi) \nabla H(\nabla H^o(\xi) ) = H(\xi) \nabla
H^o(\nabla H(\xi) ) = \xi,\quad \forall \xi \in
\R^n\setminus \{0\}. \notag
\end{gather}
\begin{definiz}[Anisotropic perimeter] Let $K$ be an
  open bounded set of $\mathbb R^n$ with Lipschitz boundary. The
  perimeter of $K$ is defined as the quantity
\[
P_H(K) = \int_{ \partial K} H(\nu_K) d\mathcal H^1.
\]
\end{definiz}

The anisotropic perimeter of a set $K$ is
finite if and only if the usual Euclidean perimeter $P(K)$ is
finite. Indeed, by properties \eqref{eq:omo} and \eqref{eq:lin} we
have that
\[
\frac{1}{\beta} |\xi| \le H^o(\xi) \le \frac{1}{\alpha} |\xi|,
\]
and then
\begin{equation*}\label{eq:per}
\alpha P(K) \le P_H(K) \le \beta P(K).
\end{equation*}

An isoperimetric inequality for the anisotropic perimeter holds,
namely
\begin{equation}
  \label{isop}
  P_H(K) \ge n\kappa_n^{\frac 1 n} |K|^{1-\frac 1 n}
\end{equation}
(see for example
\cite{bu}, \cite{dpf}, \cite{fomu}, \cite{aflt}). We stress that in
\cite{dpg1} an isoperimetric inequality for the
anisotropic relative perimeter in the plane is studied. 

\subsection{Quermassintegrals and the Steiner formula in the
  anisotropic case}
We recall some basic tools of convex analysis, well-known in the
Euclidean case.

Let $K$ be a convex body, and denote with $W^H_j(K)$,
$0\le j \le n$, the value
\[
W^H_j(K)=V\big[\underbrace{\mathcal W,\ldots,\mathcal W}_{j\;\rm
  times},\underbrace{K,\ldots,K}_{n-j\;\rm times}\big],
\]
where the right-hand side denotes the $j$-th mixed volume of order $n$
(see \cite{schn} and \cite{and}). In particular, $W^H_0(K)=|K|$,
$nW^H_1(K)=P_H(K)$ and $W^H_n(K)=\kappa_n$. In the Euclidean case,
$W^H_j$ is known as $j-$th quermassintegral of $K$. The monotonicity
properties of the mixed volumes (see \cite{schn}) give that
$W^H_j(K)$ is monotone increasing with respect to the inclusion of
convex sets.

Moreover, for $\delta>0$, the following Steiner formulas hold
(\cite{and}, \cite{schn}):
\begin{equation}
  \notag
|K+\delta \mathcal W|= |K| + P_H(K) \delta + \binom{n}{2} W^H_2(K) \delta^2
+\ldots + \binom{n}{p} W^H_p(K)\delta^p +\ldots+ \kappa_n\delta^n,
\end{equation}
and
\begin{multline}
\label{st2}
P_H(K\!+\!\delta \mathcal W) =\\= P_H(K) \! + \!
n(n-1) W^H_2(K) \delta + \!\ldots\! + n \binom{n-1}{p} W^H_{p+1}(K)\delta^p
+\!\ldots\!+n\kappa_n\delta^{n-1}.
\end{multline}

Formula \eqref{st2} immediately gives that
\begin{equation}
  \label{steinercons}
  \lim_{\delta\rightarrow 0^+}\frac{P_H(K+\delta \mathcal W)-P_H(K)}{\delta}=
  n(n-1)W_2^H(K).
\end{equation}

Finally, we recall the Aleksandrov-Fenchel inequalities, stated for
$W^H_j$ (see \cite[Section 2.7]{and}):
\begin{equation}
  \label{afineq}
\left( \frac{W^H_j(K)}{\kappa_n} \right)^{\frac{1}{n-j}} \ge \left(
  \frac{W^H_i(K)}{\kappa_n} \right)^{\frac{1}{n-i}}, \quad 0\le i < j
\le n-1,
\end{equation}
and the equality in \eqref{afineq} holds if and only if $K$ is
homothetic to $\mathcal W$.

In particular, for $i=0$ and $j=1$ we obtain the isoperimetric
inequality \eqref{isop}. Moreover, if $i=1$ and $j=2$, we have
\begin{equation*}
  \label{afineq2}
  W_2^{H}(K) \ge \kappa_n^{\frac{1}{n-1}} \left( \frac{P_H(K)}{n}
  \right)^{\frac{n-2}{n-1}}.
\end{equation*}

Finally, we recall that, under some particular assumptions on $H$,
when 
$K$ has $C^{1,1}$ boundary then
\[
W_2^{H}(K)= \int_{\de K} \kappa_H^K d P_H,
\]
where $\kappa_H^K$ is the anisotropic mean curvature of $K$ (see \cite{bccn}
for more details).

\subsection{Eigenvalue problems for $\mathcal Q_p$}
\label{seceig}
We deal with operators whose prototype is the following:
\begin{equation*}
  \label{intro:p=2}
  \mathcal Q_p v:= \sum_{i=1}^{n} \frac{\de}{\de x_i}
  \big(H(Dv)^{p-1}H_{\xi_i}(Dv)\big),
\end{equation*}
with $1<p<+\infty$. We suppose that $H$ verifies the hypotheses of the
subsection \ref{anintro}, assuming also that $H^p(\xi)$ is a
strictly convex function. 

In all this subsection we will denote with $\Omega$ a bounded
connected open set of $\R^n$.
Let us consider the eigenvalue problem associated to the operator
$\mathcal Q_p$ in $\Omega$, namely 
\begin{equation}
  \label{autq}
  \left\{
    \begin{array}{ll}
      -\mathcal Q_p u=\lambda |u|^{p-2}u &\text{in } \Omega,\\
      u=0 &\text{on } \de\Omega.
    \end{array}
  \right.
\end{equation}
In \cite{bfk} it has been proved the following result:
\begin{theo}
  There exists the first eigenvalue of \eqref{autq}, namely
  $\lambda_p(\Omega)>0$, and it is simple. Moreover, the first
  eigenfunctions have a sign and belong to $C^{1,\alpha}$. Finally,
  the following variational formulation holds: 
\begin{equation}
  \label{carvar}
  \lambda_p(\Omega)=\min_{u\in W_0^{1,p}(\Omega)\setminus\{0\}} \frac{\int_\Omega
    H(Du)^pdx}{\int_\Omega |u|^pdx}.
\end{equation}
\end{theo}
On the other hand, in \cite{aflt} it is proved a P\'olya-Szeg\"o
principle related to $H$. In order to recall this result, we denote
with $\Omega^\star$ the Wulff shape centered at the origin with the
same Lebesgue measure as $\Omega$, and $v^\star$ the so-called
convex rearrangement of $v$ with respect to $H$, that is the level
sets of $v^\star$ have the same measure as the level sets of $v$ and
they are homothetic to the Wulff shape.
\begin{theo}\label{pzani}
  If $v\in W_0^{1,p}(\Omega)$, then
\begin{equation}
\label{pz}
\int_\Omega H(Dv)^pdx \ge \int_{\Omega^\star} H(Dv^\star)^pdx.
\end{equation}
\end{theo}
Moreover, the equality case is characterized in the following result,
contained in \cite{etpolya} and \cite{fvpolya}:
\begin{theo}
  If $u$ is a nonnegative function in $W_0^{1,p}(\Omega)$ such that
  \[
  \left| \{|Du^\star|=0\}\cap \{ 0<u^\star<\|u\|_{\infty} \}  \right|=0,
    \]
    and
    \begin{equation}
      \label{pz=}
      \int_\Omega H(Du)^pdx = \int_{\Omega^\star} H(Du^\star)^pdx,
    \end{equation}
    then, up to translation, $u=u^\star$ a.e. in $\Omega$ and
    $\Omega=\Omega^\star$.
\end{theo}
An immediate consequence of Theorem \ref{pzani} is the Faber-Krahn
inequality for $\mathcal Q_p$ (see again \cite{bfk}):
\begin{equation*}
\lambda_p(\Omega) \ge \lambda_p(\Omega^\star).
\end{equation*}

As matter of fact,  by standard argument it is possible to show that
when $\Omega$ is homothetic to a Wulff shape, the eigenfunctions
related to $\lambda_p$ inherit some symmetry properties. More
precisely, we have the following result.
\begin{theo}
Let $v\in C^{1,\alpha}$ a positive solution of 
\begin{equation}
  \label{autwulf}
  \left\{
    \begin{array}{ll}
      -\mathcal Q_p v=\lambda_p |v|^{p-2}v &\text{in } \mathcal W_R,\\
      v=0 &\text{on } \de  \mathcal W_R.
    \end{array}
  \right.
\end{equation}
Then, there exists a function $\varphi(r)$, $r=H^o(x)$, and $x \in 
\mathcal W_R$, such that 
\begin{equation*}\left\{
   \begin{array}{lr}
v(x)=\varphi(r), & r \in [0,R]\\
\varphi>0\text{ in }[0,R[,&\varphi(R)=0,\\
\varphi'<0\text{ in }]0,R],&\varphi'(0)=0.
\end{array}
\right.
\end{equation*}
\end{theo}
\begin{proof}
  Let $v$ be a first eigenfunction of \eqref{autwulf} such that
  $\|v\|_{L^p}=1$. Then, by simplicity of $\lambda_p$, $v$ is unique and
\begin{equation*}
  \lambda_p(\mathcal W_R)=\int_{\mathcal W_R} H(Dv)^p dx.
\end{equation*}
By \eqref{carvar}, \eqref{pz}, \eqref{pz=} and regularity of $v$, it
follows that $v=v^{\star}$ and then $v$ is a symmetrically decreasing
function with respect to $H^o$, and this gives the thesis. 
\end{proof}
\subsection{Anisotropic $p$-torsional rigidity}
Let us consider the following problem:
\begin{equation}
  \label{eq:ptor2}
  \left\{
  \begin{array}{ll}
    -\mathcal Q_p u = 1 &\text{in }\Omega,\\
    u=0 &\text{on }\de\Omega,
  \end{array}
  \right.
\end{equation}
with $\Omega$ bounded open set of $\R^n$, $n\ge 2$. As before, we
suppose that $H$ verifies the hypotheses of the subsection
\ref{anintro}, assuming also that $H^p(\xi)$ is a strictly convex
function.  

Classical results  guarantee that \eqref{eq:ptor2} admits a solution
$u_p\in W_0^{1,p}(\Omega)$. Moreover, such solution is unique. Indeed,
suppose that $u_1$ and $u_2$ are two solutions of
\eqref{eq:ptor2}. Then, for any $\varphi\in W_0^{1,p}(\Omega)$, we
have 
\[
\int_\Omega H(Du_i)^{p-1}H_\xi(Du_i)\cdot D\varphi dx =\int_\Omega
\varphi\, dx,\quad i=1,2.
\]
Choosing $\varphi=u_1-u_2$ for $i=1,2$ and subtracting we have
\[
\int_\Omega \left(H^{p-1}(Du_1)H_\xi(Du_1) -H^{p-1}(Du_2)H_\xi
  (Du_2)\right)\cdot \left(Du_1-Du_2\right) dx = 0.
\]
Hence, being $H^p$ is strictly convex, this may happen if and only if
$u_1=u_2$ in $\Omega$. As matter of fact, the solution of
\eqref{eq:ptor2} is positive in $\Omega$. 

In view of the above considerations, we define the $p$-torsional
anisotropic rigidity of $\Omega$ the number $\tau_p(\Omega)>0$ such
that  
\begin{equation*}
  \label{eq:ptor1}
  \tau_p(\Omega)=\int_\Omega H(Du_p)^p dx = \int_\Omega u_p dx,
\end{equation*}
where $u_p\in W_0^{1,p}(\Omega)$ is the unique solution of
\eqref{eq:ptor2}.

A characterization of the anisotropic $p$-torsional rigidity is
provided by the equality
$\tau_p(\Omega)=\sigma(\Omega)^{\frac{1}{p-1}}$, where
$\sigma(\Omega)$ is the best constant in the Sobolev inequality
\[
\|u\|_{L^1(\Omega)}^p \le \sigma(\Omega) \|H(Du)\|^p_{L^p(\Omega)},
\]
that is
\begin{equation}
  \label{tors0}
\tau_p(\Omega)^{p-1} = \sigma(\Omega)= \max_{\substack{\psi \in
    W_0^{1,p}(\Omega) \setminus\{0\} \\ \psi\ge 0}}
\dfrac{\left(\displaystyle\int_\Omega \psi \, 
    dx\right)^p}{\displaystyle\int_\Omega H(D\psi)^p dx},
\end{equation}
and the solution $u_p$ of \eqref{eq:ptor2} realizes the maximum
in \eqref{tors0}. Indeed, let $\varphi\in
W_0^{1,p}(\Omega)$, $\varphi\not\equiv 0$ a nonnegative test function for
\eqref{eq:ptor2}. Using 
the definition \eqref{eq:pol} of polar function and \eqref{eq:H1}, we
have that
\[
H_\xi(Du_p)\cdot D\varphi \le H^o(H_\xi(Du_p))H(D\varphi)=H(D\varphi).
\]
Hence, it follows that
\begin{multline}
  \label{polyaequiv}
\int_\Omega \varphi\,dx = \int_\Omega H^{p-1}(Du_p) H_\xi(Du_p)\cdot
D\varphi\,dx \le \\ \le \left(\int_\Omega H(Du_p)^p\,dx\right)^{\frac {p-1}
  {p}} \left(\int_\Omega H(D\varphi)^p dx
\right)^{\frac 1 p}= \tau_p(\Omega)^{\frac{p-1}{p}} \left(\int_\Omega
  H(D\varphi)^p dx 
\right)^{\frac 1 p}.
\end{multline}
Clearly, the function $u_p$ verifies the equality in
\eqref{polyaequiv}, and then \eqref{tors0} holds. 

A consequence of the anisotropic P\'olya-Szeg\"o inequality (Theorem
\ref{pzani}) is the following upper bound for $\tau_p(\Omega)$.
\begin{theo}
  Let $\Omega$ be a bounded open set of $\mathbb R^n$. Then,
  \begin{equation}
    \label{uptor}
  \tau_p(\Omega) \le \tau_p(\Omega^\star),
  \end{equation} 
  where $\Omega^\star$ is the Wulff shape centered at the origin with
  the same Lebesgue measure as $\Omega$. 
\end{theo}
\begin{proof}
  The characterization of $\tau_p(\Omega)$ in \eqref{tors0}, the
  inequality \eqref{pz} and the properties of rearrangements give that
  \[
 \tau_p(\Omega)^{p-1}= \frac{\left(\displaystyle\int_\Omega u_p\,
     dx\right)^p}{\displaystyle\int_\Omega H(Du_p)^p\, dx} \le
 \frac{\left(\displaystyle\int_\Omega 
     u_p^\star\, dx\right)^p}{\displaystyle \int_\Omega
   H(Du_p^\star)^p\, dx} \le \tau_p(\Omega^\star)^{p-1}, 
 \]
 where $u^\star$ is the convex rearrangement of $u$ with respect to
 $H$. 
\end{proof}
\begin{rem}
  Due to the symmetry of the problem, the value of
  $\tau_p(\Omega^\star)$ can be explicitly calculated, obtaining from
  \eqref{uptor} that 
  \[
  \tau_p(\Omega)\le
  \frac{1}{n^{p'}\kappa_n^{p'/n}}\frac{n}{p'+n}|\Omega|^{\frac{p'}{n}+1},
  \]
  where $p'=p/(p-1)$ and $\kappa_n$ is the measure of the Wulff shape
  $\mathcal W$.
\end{rem}

\subsection{$p$-circular functions}
\label{pcirc}
Now we recall the definition and some basic properties of certain
generalized trigonometric functions called $p$-circular functions. We
refer the reader, for example, to \cite{bbcdg,l95,rw99,dp87}.

Let us consider the function $F_p\colon [0,(p-1)^{1/p}] \to \R$
defined as
\[
F_p(x)=\int_0^x \frac{dt}{[1-t^p/(p-1)]^{1/p}}.
\]
Denote by $z(s)$ the inverse function of $F$ which is defined on the
interval $\Big[0, \frac{\pi_p}{2}\Big]$, where
  \begin{equation}
\label{defpp}
\pi_p=2\int_0^{(p-1)^{1/p}} \frac{dt}{[1-t^p/(p-1)]^{1/p}}.
\end{equation}
The $p$-sine function, $\sin_p$ is the following periodic extension of
$z(t)$:
\[
\sin_p t= \left\{
    \begin{array}{ll}
z(t), &  t \in \Big[0,\frac{\pi_p}{2}\Big],\\[.15cm]
{z\left({\pi_p}-t\right)}, &  t \in \Big[\frac{\pi_p}{2},\pi_p
\Big],\\[.15cm]
- \sin_p(-t), & t \in \Big[-\pi_p, 0\Big],
\end{array}
\right.
\]
and it is extended periodically to all $\R$, with period $2\pi_p$.

Hence $\sin_p$ is a  odd $2\pi_p$-periodic function defined on the
whole real line and it coincides with the usual sine function when $p=2$.

The $p$-cosine function is defined as follows
\[
\cos_p t= \sin_p\left(t+ \frac{\pi_p}{2}\right),
\]
and it is an even $2 \pi_p$-periodic function.

The study of $p-$circular functions is motivated also by the
connection  with the one-dimensional $p-$Laplacian. Indeed, the
function $u(t)= {a}{\lambda^{-1/p}} \sin_p \lambda^{1/p}(t-t_0)$
is the unique solution of the problem 
\[
\left\{
    \begin{array}{l}
-\left( |u'|^{p-2}u'\right)'=\lambda |u|^{p-2}u, \\[.15cm]
u(t_0)=0, \quad u'(t_0)=a.
\end{array}
\right.
\]

On the other hand, if we consider the minimum problem
\[
\lambda=\min_{\substack{\phi\in W^{1,p}([0,\pi_p/2])\\
    u'(0)=0,\,u(\frac{\pi_p}{2})=0 }} \frac{\int_0^{\frac{\pi_p} {2}} |\phi'|^p dt
}{\int_0^{\frac{\pi_p} {2}} |\phi|^p dt }, 
\]
the value $\lambda$ is reached by the solutions of the problem
\[
\left\{
    \begin{array}{l}
      -\left( |u'|^{p-2}u'\right)'=\lambda |u|^{p-2}u, \\[.15cm]
    u'(0)=0,\quad  u(\frac{\pi_p}{2})=0.
    \end{array}
\right.
\]
Then, using the explicit expression of the solutions, it is easy to
show that
\[
\lambda_k=(2k+1)^p,\quad k\in \mathbb N_0.
\]
Hence, $\lambda_0=1$ is the first eigenvalue, reached at the functions
$\phi(t)=a\sin_p\left(t-\frac{\pi_p}{2}\right)=-a\cos_p t$, for any
constant $a$. Such computations will be useful in the next sections.

Finally, we recall that $\left(\frac{\pi_p}{L}\right)^p$ is the first
eigenvalue of the $p$-Laplacian on $[0,L]$, without distinction
between the Neumann and the Dirichlet conditions.

\section{Some useful preliminary results}
Let $\Omega$ be a bounded convex open set of $\R^n$, and $d_H(x)$ the
anisotropic distance of a point $x\in \Omega$ to the boundary $\de
\Omega$, that is
\[
d_H(x)= \inf_{y\in \de \Omega} H^o(x-y).
\]
By the property \eqref{eq:H1}, the distance function $d_H(x)$
satisfies
\begin{equation*}
  H(D d_H(x))=1.
\end{equation*}
Finally, we observe that the convexity of $\Omega$ gives that $d_H(x)$
is concave.

For further properties of the anisotropic distance function we refer
the reader to \cite{cm07}.

We denote by
\[
\Omega_t=\{ x\in\Omega\colon d_H(x)>t\}, \quad t\in [0,r_\Omega],
\]
where $r_\Omega$ is the anisotropic inradius of $\Omega$, that is
$r_\Omega= \sup \left\{r>0\colon \mathcal W_r(x_0) \subset \Omega,
  x_0\in \Omega\right\}$.

The general Brunn-Minkowski Theorem (see \cite[p.339]{schn}) and the
concavity of the an\-isotropic distance function give that the function
$P_H(\Omega_t)$ is concave in $[0,r_\Omega]$. Hence, the function
$P_H(\Omega_t)$, $t\in [0,r_\Omega]$ is a decreasing and absolutely
continuous.
\begin{lemma}
  \label{derquer}
  For almost every $t\in ]0,r_\Omega[$,
  \begin{equation}
    \label{lemmader0}
    -\frac{d}{dt} |\Omega_t| = P_H(\Omega_t),
  \end{equation}
  and
  \begin{equation}
    \label{lemmader}
    -\frac{d}{dt} P_H(\Omega_t) \ge n(n-1)W^H_2(\Omega_t),
  \end{equation}
  where the equality sign in \eqref{lemmader} holds if $\Omega$ is
  homothetic to a Wulff shape.
\end{lemma}
\begin{proof}
 We first recall that the function $d_H$ belongs to $W^{1,\infty}(\Omega)$, and
  $H(Dd)=1$ almost everywhere. Hence, using the definition of $P_H$
  and the coarea formula, for almost every $t\in
  ]0,r_\Omega[$ we have that
  \begin{equation*}
P_H(\Omega_t)= \int_{\{d_H=t\}} H\left( \frac{D d_H}{|D d_H|}\right) dx =
\int_{\{d_H=t\}} \frac{1}{|D d_H|} dx = -\frac{d}{dt} |\Omega_t|,
  \end{equation*}
  that is \eqref{lemmader0}.

  In order to show \eqref{lemmader}, it is not difficult to prove that
  \[
  \Omega_t + \rho\mathcal W \subset \Omega_{t-\rho},\quad 0<\rho<t,
  \]
  and the equality holds when $\Omega$ is homothetic to a Wulff
  shape. Since the perimeter is monotone with respect to the inclusion
  of convex sets, the above relation and \eqref{steinercons} give
  that
  \begin{multline*}
     -\frac{d}{dt} P_H(\Omega_t) = \lim_{\rho\rightarrow 0^+}
 \frac{ P_H(\Omega_{t-\rho}) - P_H(\Omega_t) }{ \rho } \ge \\ \ge
 \lim_{\rho\rightarrow 0^+}
 \frac{P_H( \Omega_t + \rho \mathcal W) - P_H( \Omega_t ) }{ \rho } =
 n(n-1) W^H_2(\Omega_t),
  \end{multline*}
 for almost every  $t\in]0,r_\Omega[$.
\end{proof}

An immediate consequence of Lemma \ref{derquer} and that $H(Dd_H)=1$
is the following result.
\begin{lemma} \label{derquer2}
  Let $u(x)=f(d_H(x))$, where $f\colon[0,+\infty[\rightarrow[0,+\infty[$
  is a strictly $C^1$ function with $f(0)=0$. Set
  \[
  E_t=\{x\in\Omega\colon u(x)>t\}=\Omega_{f^{-1}(t)}.
  \]
  Then, for a.e. $t\in]0,r_\Omega[$,
  \[
  -\frac{d}{dt} P_H(E_t) \ge n(n-1) \frac{W^H_{2} (E_t)}{H(Du)_{u=t}}.
  \]
\end{lemma}

\section{Upper bounds for $\lambda_p(\Omega)$}

\subsection{P\'olya-type estimates}

\begin{theo}\label{polythm}
    Let $\Omega$ be a bounded, convex, open set of $\R^n$. Then,
    \begin{equation}
      \label{thesis}
  \lambda_p(\Omega) \le \left(\frac{\pi_p}{2}\right)^p
  \left( \frac{P_H(\Omega)}{|\Omega|} \right)^p,
\end{equation}
where $\pi_p$ is given in \eqref{defpp}.
\end{theo}
\begin{proof}
  Let $g(t)=g(d_H(x))$, where $g$ is a nonnegative, increasing,
  sufficiently smooth function. Moreover, suppose that $g(0)=0$. We
  observe that, being  $H(d_H)=1$, by the
  coarea formula we have
  \[
  \int_\Omega H(Dg(d_H(x)))^p dx = \int_0^{r_\Omega} g'(t)^p dt \int_{d_H> t}
  \frac{1}{|Dd_H|} d\mathcal H^{n-1} = \int_0^{r_\Omega} g'(t)^p
  P_H(\{d_H>t\})dt,
  \]
  and, similarly,
  \[
  \int_\Omega g(d_H(x))^p dx = \int_0^{r_\Omega} g(t)^p P_H(\{d_H>t\})dt.
  \]
  Hence, denoting with $P(t)=P_H(\{x\in \Omega\colon d_H>t\})$,
  using $g(d_H(x))$ as test function in the Rayleigh quotient
  of \eqref{carvar}, we have
  \begin{equation}
    \label{polyacarvar}
  \lambda_p(\Omega) \le \frac{\int_0^{r_\Omega} g'(t)^p
    P(t)\,dt}{\int_0^{r_\Omega} g(t)^p P(t)\,dt}.
  \end{equation}
  Now, we perform the change of variable
  \[
  s =  \frac{C}{|\Omega|} A(t),
  \]
  where we have denoted, for the sake of simplicity, $A(t)=|\{x\in
  \Omega\colon d_H(x)>t\}|$, and $C$ is a positive constant which will
  be chosen in the next. Observe that $A(0)=|\Omega|$.

  Let $h(s)$ be the function such that
  \[
  h(s)=g(t).
  \]
  We stress that $h(C)=0$ and $h(s(t))$ is decreasing, being $g(0)=0$
  and $g(t)$ increasing. Then, substituting $h$ in
  \eqref{polyacarvar} and recalling \eqref{lemmader0} it follows that
  \begin{equation*}
  \lambda_p(\Omega) \le \frac{C^p}{|\Omega|^p}
  \frac{\displaystyle\int_0^{r_\Omega} \left[-h'\left(
      \frac{C}{|\Omega|}A(t)\right)\right]^p P(t)^p
    \left[\frac{C}{|\Omega|} P(t) \right] dt}
  {\displaystyle\int_0^{r_\Omega} \left[ h \left(
        \frac{C}{|\Omega|}A(t)\right)\right]^p
  \left[\frac{C}{|\Omega|} P(t) \right] dt}
\end{equation*}
and then
\begin{equation}
  \label{fine}
  \lambda_p(\Omega) \le
  \frac{P_H(\Omega)^p}{|\Omega|^p} C^p
  \frac{\int_0^{C} [-h'(s)]^p ds}{\int_0^{C} [h(s)]^p ds}.
\end{equation}
Now we want to choose $C$ and $h$ in order to minimize the quantity in
the right-hand side of \eqref{fine}. As matter of fact, such minimum
does not depend on $C$, being
\[
C^p \frac{\int_0^{C} [-h'(s)]^p ds}{\int_0^{C} [h(s)]^p ds} =
\frac{\int_0^{1} [-\tilde h'(\sigma)]^p d\sigma}{\int_0^{1} [\tilde
  h(\sigma)]^p d\sigma}, 
\]
where $\tilde h(\sigma)=h(C\sigma)$. Hence, choosing
$C=\frac{\pi_p}{2}$, where $\pi_p$ is defined in \eqref{defpp} we get
\begin{equation}
  \label{minpb}
  \lambda_p(\Omega) \le \left( \frac{\pi_p}{2} \right)^p
\frac{P_H(\Omega)^p}{|\Omega|^p} \min_{\mathcal A}
\frac{\int_0^{\frac{\pi_p}{2}} [-h'(s)]^p
  ds}{\int_0^{\frac{\pi_p}{2}} [h(s)]^p ds},
\end{equation}
where $\mathcal A$ is the class of positive decreasing functions $h\in
W^{1,p}([0,\frac{\pi_p}{2}])$ such that $h(\frac{\pi_p}{2})=0$. 
On the other hand, using the properties of the $p$-circular function
(see Section \ref{pcirc}) we have that
\[
\min_{\mathcal A} \frac{\int_0^{\frac{\pi_p}{2}} [-h'(s)]^p
  ds}{\int_0^{\frac{\pi_p}{2}} [h(s)]^p ds} \le
\frac{\int_0^{\frac{\pi_p}{2}} [-\cos_p(s)']^p
  ds}{\int_0^{\frac{\pi_p}{2}} [\cos_p(s)]^p ds}=1.
\]
Hence,
\begin{equation*}
  \lambda_p(\Omega) \le \left( \frac{\pi_p}{2} \right)^p
\frac{P_H(\Omega)^p}{|\Omega|^p},
\end{equation*}
and the proof is completed.
\end{proof}
\begin{rem}
  \label{remopt}
  We observe that in the Euclidean case, with $H(\xi)=|\xi|$, the
  estimate \eqref{thesis} gives an upper bound for the first
  eigenvalue of the $p$-Laplace
  operator. Moreover, we stress that, for a general norm $H$, the
  constant $\left(\frac{\pi_p}{2}\right)^p$ does not depend on $H$.
\end{rem}
\begin{rem}
  We stress that, in the case $p=2$, a direct application of the estimate
  \eqref{pol}, proved in \cite{po61} and quoted in the introduction,
  allows to obtain a result analogous to the inequality
  \eqref{thesis}. Indeed, by property \eqref{eq:lin}, if $w$ is an
  eigenfunction of the 
  Dirichlet-Laplacian related to the first eigenvalue
  $\lambda_{1,2}(\Omega)$, we have 
  \[
  \lambda_2(\Omega) \le \frac{\int_\Omega H(Dw)^2dx}{ \int_\Omega |w|^2
    dx} \le \beta^2 \frac{\int_\Omega |Dw|^2dx}{ \int_\Omega |w|^2
    dx} = \beta^2 \lambda_{1,2}(\Omega) \le \frac{\pi^2}{4} 
  \left( \frac{\beta P(\Omega)}{|\Omega|}\right)^2.
  \]
 Then, recalling that $P_H(\Omega) \le \beta P(\Omega)$, the above
 estimate for $\lambda_2(\Omega)$ gives a larger bound than
 \eqref{thesis}.  
\end{rem}
\begin{rem}
\label{optaut}
  We show that in a particular case, the estimate \eqref{thesis} is
  optimal.  Let $n=2$, and $H(x,y)=(|x|^p+|y|^p)^{1/p}$. Hence the
  operator $\mathcal Q_p$ is the pseudo-$p$-Laplacian, that is
  \[
  \mathcal Q_p u= - \left(
    |  u_x |^{p-2} u_x \right)_x - \left(
    |  u_y |^{p-2} u_y \right)_y.
  \]
  Choosing $\Omega=[0,a]\times [0,b]$, and observing that
  $P_H(\Omega)=2(a+b)$, the inequality \eqref{thesis}
becomes
  \begin{equation}
  \label{rett}
  \lambda_p(\Omega) = \frac{
    \int_0^b\int_0^a \left( |\bar u_x|^p +  |\bar u_y |^p \right )dxdy}
    {\int_0^b \int_0^a |u|^p dx dy } \le \pi_p^p
\left(\frac 1 a+ \frac 1 b\right)^p,
\end{equation}
where  $\bar u$ is a first eigenfunction. As matter of fact,
considering the Rayleigh quotient related to $\lambda_p(\Omega)$, we have
\begin{multline}
\label{quo}
\lambda_p(\Omega) \ge \frac{ \int_0^b\int_0^a \left( |\bar
    u_x|^p\right)dxdy} {\int_0^b \int_0^a |\bar u|^p dx dy } 
\ge \min_{y\in [0,b]} \frac{ \int_0^a \left( |\bar u_x|^p \right )dx}
{\int_0^a |\bar u|^p dx }\ge \\ 
\ge \min_{y\in [0,b]}\min_{\phi \in W_0^{1,p}([0,a])\setminus\{0\}} \frac{
  \int_0^a \left( |\phi'|^p \right )dx} {\int_0^a |\phi|^p dx
}=\frac{\pi_p^p}{a^p}. 
\end{multline}
We explicitly observe that the second inequality in \eqref{quo}
follows form the general inequality
\begin{equation*}
\inf_{t\in[c,d]} \frac{f(t)}{g(t)} \le \frac{\int_c^d
  f(t)\,dt}{\int_c^d g(t)\,dt} \le \sup_{t\in[c,d]} \frac{f(t)}{g(t)},
\end{equation*}
where $f$ and $g$ are two nonnegative integrable functions in $[c,d]$,
with $g\ne 0$. 

Joining \eqref{rett} and \eqref{quo} we get
\begin{equation*}
\frac{\pi_p^p}{a^p} \le \lambda_p(\Omega)\le \pi_p^p
\left(\frac 1 a+ \frac 1 b\right)^p.
\end{equation*}
Letting $b\rightarrow +\infty$, we get the required optimality. Hence,
this example proves that the constant in \eqref{thesis} cannot be
improved, in general, for any norm $H$ of $\R^n$. 
\end{rem}
\subsection{Stability}
\begin{theo}
  Let $\Omega$ be a bounded, convex, open set of $\R^n$. Then, we have
  that
  \begin{equation}
    \label{quap}
 \frac{ \lambda_p(\Omega) -
   \lambda_p(\tilde \Omega)}{\lambda_p(\Omega)} \le
C_\Omega
\left(1-\frac{n^{\frac{n}{n-1}}\kappa_n^{\frac{1}{n-1}}|\Omega|}
  {P_H(\Omega)^{\frac{n}{n-1}}} \right),
\end{equation}
where $\tilde \Omega$ is the Wulff shape centered at the origin such
that $P_H(\tilde \Omega)=P_H(\Omega)$, and
\[
C_\Omega= \frac{\|v\|_{\infty}^p }{\|v\|_p^p} |\tilde \Omega|,
\]
where $v$ is an eigenfunction of $\mathcal Q_p$ in $\tilde \Omega$.
\end{theo}
\begin{proof}
    Without loss of generality,
 we can suppose that the quantity in the right-hand side of
 \eqref{quap} is smaller than $1$. Otherwise, \eqref{quap} is
 trivial.
 Let $R>0$ such that $\mathcal W_R=\tilde \Omega$, and define
 \[
 u(x)= \varphi(R-d_H(x)), \quad x\in\Omega,
 \]
 where $d_H$ is the  anisotropic distance function to the boundary of
 $\Omega$,  $\varphi(H^o(x))=\varphi(r)=v(x)$, and $v$ is a fixed
 positive eigenfunction of $\mathcal Q_p$ in $\tilde \Omega$,  (see
 Section \ref{seceig} for the details).
 As matter of fact, being $v$ simmetric with respect to $H^o$, we can
 define the function $g(t)=H(Dv)|_{\{v=t\}}$, $0\le
 t \le \|v\|_\infty$. Hence, by construction, the function $u$ has the
following properties:
\begin{equation*}\left\{
    \begin{array}{l}
u(x)\in W_0^{1,p}(\Omega),\\
\left. H(Du)\right|_{\{u=t\}} = g(t), \\
\|u\|_{\infty}\le \|v\|_{\infty}.
\end{array}
\right.
\end{equation*}
 We put
 \[
 E_t=\{x \in \Omega \colon u>t \},\quad \mathcal B_t=\{x \in
 \tilde\Omega \colon v > t \},
 \]
 and observe that $E_t$ and $\mathcal B_t$ are convex sets of
 $\R^n$. Lemma \eqref{derquer2} and the Aleksandrov-Fenchel
 inequalities imply
  \[
  -\frac {d}{dt} P_{H}(E_t) \ge n(n-1)\frac{W_2^H(E_t)}{g(t)} \ge
  (n-1)n^{\frac{1}{n-1}}\kappa_n^{\frac{1}{n-1}}
  \frac{P_H(E_t)^{\frac{n-2}{n-1}}}{g(t)},
  \]
  while
 \[
  -\frac {d}{dt} P_{H}(\mathcal B_t) = n(n-1)\frac{W_2^H(\mathcal
    B_t)}{g(t)} = (n-1)n^{\frac{1}{n-1}}\kappa_n^{\frac{1}{n-1}}
  \frac{P_H(\mathcal B_t)^{\frac{n-2}{n-1}}}{g(t)}.
  \]
  Together with the initial condition $P_H(E_0)=P_H(\mathcal B_0)$, we
  have that
  \begin{equation}
    \label{pp}
  P_{H}(E_t) \le P_{H} (\mathcal B_t), \quad 0<t<\| u\|_\infty.
  \end{equation}
   Now, denote with $\mu(t)=|E_t|$ and $\nu(t)=|\mathcal B_t|$. Using
   the coarea formula and the above inequality, we have that, for
   almost every $0\le t
   <\|u\|_{\infty}$,
   \[
  \begin{aligned}
    -\mu'(t)=\int_{ \{u=t\}} \frac{1}{|Du|} d\mathcal H^{n-1} &=
  \frac{1}{g(t)} \int_{\{u=t\}}  H \left( \frac{Du}{|Du|} \right) d\mathcal
  H^{n-1} = \\ &=
  \frac{P_H(E_t)} {g(t)} \le \frac{P_H(\mathcal B_t)}{g(t)}
  =\int_{\{v=t\}} \frac{1}{|Dv|}d\mathcal H^{n-1} =-\nu'(t),
\end{aligned}
\]
and then $\nu-\mu$ is a decreasing function.
Hence,
\begin{equation}
  \label{stimeden}
  \begin{aligned}
\int_{\Omega} u^p dx &=\int_0^{\|u\|_{\infty}} p t^{p-1} \mu(t)
dt =\\ &= \int_0^{\|v\|_{\infty}} p t^{p-1} \nu(t) dt
-\int_0^{\|v\|_{\infty}} p t^{p-1} \left[\nu(t) - \mu(t) \right] dt
\ge \\ &\ge \int_{\tilde\Omega} v^p dx - \left(
  |\tilde \Omega|-|\Omega|\right)\| v \|_{\infty}^p.
\end{aligned}
\end{equation}
On the other hand, by the coarea formula and \eqref{pp} we get that
\begin{equation}
    \label{stimenum}
\begin{aligned}
  \int_{\Omega} H(Du)^p dx &= \int_0^{\|u\|_{\infty}} g(t)^{p-1} \left[
  \int_{u=t} H\left(\frac{Du}{|Du|}\right) d\mathcal H^{n-1} \right]
dt =\\ &=
 \int_0^{\|u\|_{\infty}}  g(t)^{p-1} P_H(E_t) dt \le \\ &\le
 \int_0^{\|v\|_{\infty}} g(t)^{p-1} P_H(\mathcal B_t) dt =
 \int_{\tilde \Omega} H(Dv)^p dx.
\end{aligned}
\end{equation}
Hence, using \eqref{stimeden} and \eqref{stimenum} in the Rayleigh
quotient of $\mathcal Q_p$ for $u$, we have that
\begin{multline*}
  \lambda_p(\Omega) \le  \frac{\int_{\Omega} H(Du)^p dx}{
    \int_{\Omega} u^p dx} \le \frac{\int_{\tilde \Omega} H(Dv)^p
    dx}{\int_{\tilde\Omega} v^p dx - \left(
  |\tilde \Omega|-|\Omega|\right)\| v \|_{\infty}^p}=\\=
\lambda_p(\tilde \Omega) \left[ \frac{1}{1-\frac{\| v
      \|_{\infty}^p}{\| v \|_{p}^p}} \left( |\tilde
    \Omega|-|\Omega|\right) \right],
\end{multline*}
and we get the claim.
\end{proof}
\section{A lower bound for $\tau_p(\Omega)$}
In this final section we prove a lower bound for the anisotropic
$p$-torsional rigidity, namely the number $\tau_p(\Omega)>0$ such that
\begin{equation}
  \label{tors}
\tau_p(\Omega)^{p-1} = \max_{\substack{\psi\in
    W_0^{1,p}(\Omega)\setminus\{0\} \\ \psi\ge 0 }}
\dfrac{\left(\displaystyle\int_\Omega \psi \, 
    dx\right)^p}{\displaystyle\int_\Omega H(D\psi)^p dx},
\end{equation}
for a bounded convex open set $\Omega$ of $\R^n$. The following result
holds.
\begin{theo}\label{teotor}
  Let $1<p<+\infty$, and $\Omega$ as above. Then,
\begin{equation}
\label{stimator}
 \tau_p(\Omega)\ge \frac{p-1}{2p-1}
 \frac{|\Omega|^{\frac{2p-1}{p-1}}}{P_H(\Omega)^{\frac{p}{p-1}}}.
  \end{equation}

\end{theo}
\begin{proof}
 We argue as in \cite{po61}, and use the same notation and argument of
 the proof of Theorem \ref{polythm}. For a test function
 $g(t)=g(d_H(x))$, with $g$ nonnegative increasing sufficiently smooth
 function such that $g(0)=0$, we have that
 \begin{equation}
   \label{eq:5}
  \tau_p(\Omega)^{p-1} \ge \dfrac{\left(\displaystyle\int_0^{r_\Omega}g(t)
      P_H(t)dt\right)^p}{\displaystyle\int_0^{r_\Omega} g'(t)^p
    P_H(t)dt}= \dfrac{\left(\displaystyle\int_0^{r_\Omega}g'(t)
      A(t)dt\right)^p}{\displaystyle\int_0^{r_\Omega} g'(t)^p
    P_H(t)dt}.  
 \end{equation}
  Last equality follows performing an integration by parts. Substituting
  \[
  g(t)=\int_0^{t} \left(\frac{A(s)}{P_H(s)}\right)^{1/(p-1)}ds
  \]
  in \eqref{eq:5} it follows that
  \begin{multline}\label{catena}
  \tau_p(\Omega) \ge \int_0^{r_\Omega}
  \frac{A(t)^{\frac{p}{p-1}}}{P(t)^{\frac{1}{p-1}}} dt
=  \int_0^{r_\Omega}
  \frac{A(t)^{\frac{p}{p-1}}[-A'(t)]}{P(t)^{\frac{p}{p-1}}} dt
 = \\ = \frac{p-1}{2p-1}
  \frac{|\Omega|^{\frac{2p-1}{p-1}}}{P_H(\Omega)^{\frac{p}{p-1}}} +
\frac{p}{2p-1}
\int_0^{r_\Omega}
\left(\frac{A(t)}{P(t)}\right)^{\frac{2p-1}{p-1}} [-P'(t)] dt \ge \\
\ge
\frac{p-1}{2p-1} \frac{|\Omega|^{\frac{2p-1}{p-1}}}{P_H(\Omega)^{\frac{p}{p-1}}}.
\end{multline}
We precise that the second equality in \eqref{catena} follows by an
integration by parts, taking into account that, being $P(t)$
decreasing and
\[
A(t) =\int_t^{r_\Omega} P(t) dt \le P(t)(r_\Omega-t), 
\]
we have
\[
\frac{|A(t)|^{\frac{2p-1}{p-1}}}{P(t)^{\frac{p}{p-1}}}\rightarrow
0\text{ as }t\rightarrow r_\Omega.
\]
\end{proof}
\begin{rem}
  We stress that in the Euclidean case, with $H(\xi)=|\xi|$, the
  estimate \eqref{tors} extends the lower bound for the $p$-torsional
  rigidity proved in \cite{cfg02}. Moreover, we again observe that
  the constant is independent on $H$.
\end{rem}
\begin{rem}
  We stress that, similarly as observed in Remark \ref{remopt},
  it is possible to apply the inequality
  \eqref{gazz2}, known in the Euclidean case, obtaining an
  estimate analogous to \eqref{stimator}. More precisely, if
  $\tau_p(\Omega)$ is the anisotropic $p$-torsional rigidity, then
  \[
  \tau_p(\Omega) \ge \frac{p-1}{2p-1}
  \frac{|\Omega|^{\frac{2p-1}{p-1}}}{\left(\beta
      P(\Omega)\right)^{\frac{p}{p-1}}}.
  \]
Again, the above estimate gives a smaller bound than \eqref{stimator}.  
\end{rem}

\begin{rem}
We can show that in a particular case, the estimate \eqref{stimator}
is optimal.  Using the same assumptions and notations
of Remark \ref{optaut}, the inequality \eqref{stimator} becomes
  \begin{equation}
  \label{torrett}
  \tau_p(\Omega) \ge  \frac{p-1}{2p-1} \quad
  \frac{(ab)^{\frac{2p-1}{p-1}}}{2^{\frac{p}{p-1}}
    (a+b)^{\frac{p}{p-1}}}. 
\end{equation}
On the other hand, it is not difficult to prove that the $p$-torsional
rigidity in $[0,a]$ is  
\begin{equation}
  \label{torint}
  \tau_p([0,a]) =  \frac{p-1}{2p-1} \quad
  \frac{a^{\frac{2p-1}{p-1}}}{2^{\frac{p}{p-1}}}. 
\end{equation}
Moreover, reasoning as in Remark \ref{optaut}, by the definition of
$\tau_p(\Omega)$ and using the H\"older inequality we obtain
\begin{gather}
\begin{split}
\label{ott}
\tau_p(\Omega) &=  \frac{\left( \int_0^b\int_0^a u_p\,dxdy\right)^p}
{\int_0^b \int_0^a \left( |(u_p)_x|^p + |(u_p)_y|^p\right) dx dy }\le
b^{p-1}\frac{ \int_0^b\left(\int_0^a u_p \,dx\right)^p dy} {\int_0^b
  \left(\int_0^a |(u_p)_x|^p dx\right) dy } \\[.2cm] &\le  \sup_{y\in [0,b]}
b^{p-1}\frac{ \left(\int_0^a u_p \,dx\right)^p} {\int_0^a |(u_p)_x|^p
  dx } \le b^{p-1} \tau^{p-1}_p([0,a]). 
\end{split}
\end{gather}
Joining \eqref{torrett}, \eqref{ott} and \eqref{torint} we get
\begin{equation}
\label{torult}
\frac{p-1}{2p-1} \quad
\frac{(ab)^{\frac{2p-1}{p-1}}}{2^{\frac{p}{p-1}}
  (a+b)^{\frac{p}{p-1}}} \le \tau_p(\Omega)\le b \frac{p-1}{2p-1}
\quad \frac{a^{\frac{2p-1}{p-1}}}{2^{\frac{p}{p-1}}}. 
\end{equation}
Finally, \eqref{torult} gives 
\begin{equation*}
\left( \dfrac{b}{a+b}\right)^{\frac{p}{p-1}} \le \tau_p(\Omega) \le 1,
\end{equation*}
and letting $b\rightarrow +\infty$ we get the required
optimality. Hence, 
this example proves that the constant in \eqref{stimator} cannot be
improved, in general, for any norm $H$ of $\R^n$. 
\end{rem}


\end{document}